\documentclass{amsart}
\usepackage[utf8]{inputenc}
\usepackage{amssymb}
\usepackage{hyperref}
\usepackage[final]{showkeys}
\usepackage{color}

\input xy
\xyoption{all}

\newtheorem{mydef}{Definition}[section]
\newtheorem{lem}[mydef]{Lemma}
\newtheorem{lemma}[mydef]{Lemma}
\newtheorem{thm}[mydef]{Theorem}

\newtheorem{cor}[mydef]{Corollary}

\newtheorem{defin}[mydef]{Definition}
\newtheorem{definition}[mydef]{Definition}
\newtheorem{example}[mydef]{Example}
\newtheorem{remark}[mydef]{Remark}

\newtheorem{fact}[mydef]{Fact}

\newcommand{\defn}[1]{\emph{#1}}

\newcommand{\fct}[2]{{}^{#1}#2}

\newcommand{\ba}{\bar{a}}
\newcommand{\bb}{\bar{b}}
\newcommand{\bc}{\bar{c}}
\newcommand{\bd}{\bar{d}}
\newcommand{\bx}{\bar{x}}

\newcommand{\dom}[1]{\text{dom}(#1)}
\newcommand{\ran}[1]{\text{ran}(#1)}

\newcommand{\seq}[1]{\langle #1 \rangle}
\newcommand{\rest}{\upharpoonright}

\newcommand{\st}{\colon}
\newcommand{\defas}{\mathrel{\mathop:}=}

\newcommand{\bL}{\mathcal{L}}

\newcommand{\mucard}{\eta}

\newcommand{\K}{\mathbf{K}}

\newcommand{\leap}[1]{\le_{#1}}

\newcommand{\lea}{\leap{\K}}

\newbox\noforkbox \newdimen\forklinewidth
\forklinewidth=0.3pt \setbox0\hbox{$\textstyle\smile$}
\setbox1\hbox to \wd0{\hfil\vrule width \forklinewidth depth-2pt
 height 10pt \hfil}
\wd1=0 cm \setbox\noforkbox\hbox{\lower 2pt\box1\lower
2pt\box0\relax}
\def\unionstick{\mathop{\copy\noforkbox}\limits}
\newcommand{\nf}{\unionstick}
\newcommand{\nfs}[4]{#2 \nf_{#1}^{#4} #3}

\def\1nf{\unionstick^{(1)}}

\def\2nf{\unionstick^{(2)}}
\def\3nf{\unionstick^{(3)}}

\newcommand{\tp}{\text{tp}}
\newcommand{\tpqf}{\tp_{\text{qf}}}
\newcommand{\gtp}{\text{gtp}}

\newcommand{\gS}{\text{gS}}

\newcommand{\hanf}[1]{h (#1)}

\newcommand{\Axfr}{\text{AxFr}_1}

\newcommand{\Aut}{\operatorname{Aut}}

\newcommand{\cl}{\operatorname{cl}}

\newcommand{\LS}{\text{LS}}

\newcommand{\BI}{\mathbf{I}}
\newcommand{\BJ}{\mathbf{J}}

\newcommand{\Av}{\text{Av}}

\newcommand{\comment}[1]{}

\title{Categoricity in multiuniversal classes}
\date{\today }

\author{Nathanael Ackerman}
\email{nate@math.harvard.edu}
\urladdr{http://math.harvard.edu/\textasciitilde nate/}
\address{Department of Mathematics, Harvard University, Cambridge, Massachusetts, USA}

\author{Will Boney}
\email{wboney@math.harvard.edu}
\urladdr{http://math.harvard.edu/\textasciitilde wboney/}
\address{Department of Mathematics, Harvard University, Cambridge, Massachusetts, USA}
\thanks{This material is based upon work done while
  the second author was supported by the National Science Foundation under Grant No.\ DMS-1402191.}

\author{Sebastien Vasey}
\email{sebv@math.harvard.edu}
\urladdr{http://math.harvard.edu/\textasciitilde sebv/}
\address{Department of Mathematics, Harvard University, Cambridge, Massachusetts, USA}

\begin{document}

\begin{abstract}
  The third author has shown that Shelah's eventual categoricity conjecture holds in universal classes: class of structures closed under isomorphisms, substructures, and unions of chains. We extend this result to the framework of multiuniversal classes. Roughly speaking, these are classes with a closure operator that is essentially algebraic closure (instead of, in the universal case, being essentially definable closure). Along the way, we prove in particular that Galois (orbital) types in multiuniversal classes are determined by their finite restrictions, generalizing a result of the second author.
\end{abstract}

\maketitle


\section{Introduction}

One of the most important test questions in the study of abstract elementary classes (AECs) is Shelah's eventual categoricity conjecture, i.e.\ the statement that if an AEC is categorical in any ``sufficiently large'' cardinal then it must be categorical on a tail of cardinals. This conjecture, if true, would be a generalization of Morley's theorem for first order theories to AECs. 

An important class of AECs are the \emph{universal classes}: classes of structures closed under isomorphism, substructures and union of $\subseteq$-increasing chains. Essentially by a result of Tarski, universal classes are precisely the classes of models of a universal sentence in $\bL_{\infty, \omega}$ (see \cite{tarski-th-models-i}).

In a tour de force, the third author \cite{ap-universal-apal, categ-universal-2-selecta} proved Shelah's eventual categoricity conjecture for universal classes.

\begin{fact}[{\cite[Theorem 7.3]{categ-universal-2-selecta}}]
Let $\K$ be a universal class.  If $\K$ is categorical in some $\lambda \geq \beth_{\beth_{\left(2^{|\tau(\K)|+\aleph_0}\right)^+}}$, then $\K$ is categorical in all $\lambda \geq \beth_{\beth_{\left(2^{|\tau(\K)|+\aleph_0}\right)^+}}$.
\end{fact}

In this paper, we aim to generalize this result to a larger class of AECs which we call \emph{multiuniversal}, see Definition \ref{multiuniv-def}. While there are many interesting universal classes (locally finite groups, valued fields with fixed value group, etc.), the restrictions on the definition of universal classes make it difficult to have universal classes that are both mathematically interesting and model-theoretically well-behaved.  Recent work of Hyttinen and Kangas \cite{hyttinen-kangas-universal-mlq} have given some confirmation of this by building on the third author's work to show that in any universal class that is categorical in a high-enough successor, the models must eventually look like either vector spaces or be disintegrated (i.e.\ look like sets). The motivation for the present generalization stems from the goal of slightly weakening the notion of universal class while still maintaining many of the nice mathematical properties of universal classes.

In aiming to generalize beyond universal classes, we took algebraically closed fields as our prototype.  Fields are easily seen to be universal classes, but are model-theoretically intractable at that level of generality.  Algebraically closed fields are the most model-theoretically well-behaved subclass of fields, but are not universal as they are not closed under subfield. The root of the issue is the difference between definable and algebraic closure: a polynomial (like $x^2 + 2$) can have several, but only finitely-many, ``undistinguishable'' roots. Thus we define a multiuniversal class to be an AEC where any element inside the closure of a set has \emph{finitely-many} other equivalent elements: its type is algebraic (Definition \ref{alg-def}). To formalize the idea of the closure of a set, we use the notion of an AEC admitting intersections (Definition \ref{inter-def}).

Several examples beyond algebraically closed fields are listed in Example \ref{mu-example}.  A further alternate framework is to think of multiuniversal classes as universal classes that take value, not in the category of set, but in the category whose objects are sets $X$ along with the collection of finite subsets of $X$. Let us also mention that we can also characterize multiuniversal classes as certain kind of accessible categories (as in \cite{multipres-pams}), see Section \ref{top-sec}. We thank Jiří Rosický for suggesting this category-theoretic characterization.

\comment{As is well known, given any vocabulary\SEB{I changed the terminology from language to vocabulary to keep consistency with the rest of this paper and my other papers on universal classes.} $\tau$ with function symbols there is a relational vocabulary $\tau^*$ and a first order $\tau^*$-theory interdefinable with the empty theory in $\tau$. The language $\tau^*$ has all of the same relations as $\tau$ but for every $n$-ary function symbol of $\tau$, $\tau^*$ has an $(n+1)$-ary relation symbol. The corresponding $\tau^*$ theory simply says that each new $(n+1)$-ary relation symbol in $\tau$ is the graph of a function. While for most practical purposes a $\tau$-structure $M$ serves the same purpose as the corresponding $\tau^*$-structure $M^*$ where functions are replaced by their graphs, the notions of $\tau$-substructure and $\tau^\ast$-substructure do \emph{not} coincide. This is because $\tau^*$-substructures are not closed under applications of functions. 

Thus we can think of the adding of function symbols as allowing us to require that $\tau$-substructure be closed under elements which are, in some sense, uniquely definable from the substructure. This suggests that a natural weakening of the notion of a universal class would be one where we require substructures to be closed under elements which are \emph{algebraic} (i.e.\ satisfy an equation with only finitely-many solutions) over the substructure. These will be what we call \emph{multiuniversal classes} (see Definition \ref{multiuniv-def} for a more precise definition)}

The main result of the present paper is Corollary \ref{main-cor}: Shelah's eventual categoricity conjecture holds for multiuniversal AECs. The proof goes along the same lines as the third author's proof of the corresponding result for universal classes \cite{categ-universal-2-selecta}. An important step is to establish that any (not necessarily categorical) multiuniversal AEC has a strong locality property: its Galois (i.e.\ orbital) types are determined by their finite restrictions. In technical jargon, multiuniversal classes are fully $(<\aleph_0)$-tame and type-short, see Theorem \ref{tameness-thm}. For universal classes, this is due to the second author and essentially follows from the fact that a partial isomorphism from $A$ to $B$ can be extended \emph{uniquely} to the closure of $A$ under the functions of the ambient model. In the setup of multiuniversal classes, there is no longer such a unique extension (again, think of the case of algebraically closed fields of characteristic zero: an automorphism of $\mathbb{Q}$ can be extended in two ways to $\mathbb{Q} (\sqrt{-2})$). Still given any partial isomorphism $p$ from $A$ to $B$ and $a$ in the closure of $A$, there are only finitely-many extensions of $p$ with domain $\dom{p} \cup \{a\}$. This will allow us to prove the result with a compactness argument (using Tychonoff's theorem). This can also be construed as a finite-injury proof.

\subsection{Acknowledgments}

We thank the referee for a thorough report that led to a clear improvement in the presentation of the paper.

\section{Preliminaries}

We assume the reader is familiar with the basics of AECs, as covered for example in Chapters 4 and 8 of \cite{baldwinbook09}. We will make heavy use of Galois (orbital) types and use the notation from \cite[\S2]{sv-infinitary-stability-afml}. In particular, we let $\gtp (\bb / A; N)$ denote the Galois type of the sequence $\bb$ over $A$ as computed inside $N$, $\gS^\alpha (A; N)$ denotes the set of Galois types over $A$ of sequences of length $\alpha$ from $N$, and $\gS (A; N) = \gS^1 (A; N)$. We write $\K$ for an AEC, and write $\lea$ for the ordering on $\K$. When $\K$ is clear from context, we may drop it from the name of a concept. We never distinguish between a structure and its universe, as which one is meant will always be clear from context. For example, if $A$ is a set, $A \subseteq M$ means that $A$ is a subset of the universe of $M$. We may also write $\ba \in M$ to mean that $\ba$ is a sequence of elements from the universe of $M$.

The following was introduced for AECs by Baldwin and Shelah \cite[Definition 1.2]{non-locality}.

\begin{defin}\label{inter-def}
  Suppose that $\K$ is an AEC. For any $A \subseteq N \in \K$ define
  \[
  \cl^N_{\K}(A) \defas \bigcap\{ N_0 \st N_0 \lea N\text{ and }A \subseteq N_0\}.
  \] Note again that we will omit $\K$ when it is clear from context and we will not distinguish between the set $\cl^N(A)$ the the corresponding $\tau(\K)$-structure induced from the $\tau(\K)$-structure on $N$. We sometimes write $\cl^N (A \bb)$ instead of the more proper $\cl^N (A \cup \ran{\bb})$, where $\ran{\bb}$ is the set of elements of the sequence $\bb$ (i.e.\ its range when considered as a function).

We say that $\K$ is an \defn{AEC with intersections} (or \emph{$\K$ has intersections}, or \emph{$\K$ admits intersections}) if for all $N \in \K$ and $A \subseteq N$, $\cl^N(A) \lea N$. 
\end{defin}

We will heavily (and sometimes without comments) use the following basic properties of AECs with intersections \cite[Propositions 2.14 and 2.18]{ap-universal-apal}. 

\begin{fact}\label{inter-basics}
  Let $\K$ be an AEC with intersections.

  \begin{enumerate}
  \item If $f: M \to N$ is a $\K$-embedding and $A \subseteq M$, then $f [\cl^{M} (A)] = \cl^N (f[A])$. In particular, if $M \lea N$ then $\cl^M (A) = \cl^N (A)$.
  \item (Finite character) Let $M \in \K$ and let $a \in \cl^M (B)$. Then there exists a finite $B_0 \subseteq B$ such that $a \in \cl^M (B_0)$.
  \item (Equality of types) Let $M_1, M_2 \in \K$, $A \subseteq M_1 \cap M_2$. Then\footnote{Note that by itself the statement ``$A \subseteq M_1 \cap M_2$'' is somewhat meaningless: the set $A$ may be embedded in a different way inside $M_1$ and $M_2$. However both sides of the characterization of equality of types imply that $A$ is embedded in the same way in $M_1$ and $M_2$.} $\gtp (\bb_1 / A; M_1) = \gtp (\bb_2 / A; M_2)$ if and only if there exists an isomorphism $f: \cl^{M_1} (A\bb_1) \to \cl^{M_2} (A \bb_2)$ that fixes $A$ and sends $\bb_1$ to $\bb_2$.
  
  \end{enumerate}
\end{fact}

The following essentially also appears in \cite[Definition VI.1.15(2)]{shelahaecbook2}.

\begin{definition}\label{alg-def}
  In a fixed AEC $\K$, a type $p$ over a set $A$ is \emph{$\mucard$-algebraic} if for any $N \in \K$ whose universe contains $A$, the set of realizations of $p$ in $N$ has cardinality strictly less than $\mucard$. When $\mucard = \aleph_0$, we omit it and just call $p$ \emph{algebraic}.
\end{definition}

The following basic properties of algebraicity in AECs with intersections will be useful:

\begin{lem}\label{intersec-alg}
  Let $\K$ be an AEC with intersections, let $\alpha$ be an ordinal, and let $p \in \gS^\alpha (A; N)$ be a type realized in $\cl^N (A)$. 
  \begin{enumerate}
      \item For any $M \in \K$ containing $A$, if $\bb \in M$ realizes $p$, then $\bb \in \cl^M(A)$.
      \item For any $M \in \K$, if $p$ is realized in $M$, then $p$ has as many realizations in $M$ as in $N$.
  \end{enumerate}

  In particular, any type that is $\mucard$-algebraic is $\mucard_0^+$-algebraic for some $\mucard_0 < \mucard$.
\end{lem}
\begin{proof}
  Write $p = \gtp (\ba / A; N)$. For the first item, if $\bb$ realizes $p$, then there is an isomorphism $f: \cl^N(A\ba) \rightarrow \cl^M(A\bb)$ fixing $A$ and such that $f(\ba)=\bb$.  But $\cl^N(A\ba) = \cl^N(A)$ by assumption, so $\cl^M(A\bb) = f[\cl^N(A)] = \cl^M (f[A]) = \cl^M (A)$ and so $\bb \in \cl^M(A)$.

  For the second item, suppose that $\bb \in M$ realizes $p$. Then $\bb \in \cl^M(A)$ by the first item and there exists an isomorphism $f: \cl^N (A) \rightarrow \cl^M (A)$ sending $\ba$ to $\bb$ and fixing $A$. Without loss of generality, $N = \cl^N (A)$ and $M = \cl^M (A)$. Now as $f$ is a bijection, it must send distinct realizations of $p$ in $N$ to distinct realizations of $p$ in $M$ and vice-versa. The result follows.
\end{proof}

Using algebraicity, we can give the desired semantic definition of multiuniversal classes:

\begin{definition} \label{multiuniv-def}
  We say that an AEC $\K$ is \defn{multiuniversal} if it admits intersections and for any $M \in \K$, any $A \subseteq M$, and any $b \in \cl^M(A)$, $\gtp(b / A; M)$ is algebraic. 
\end{definition}

\begin{remark}\label{eta-multi-rmk}
  More generally, define an AEC with intersections to be \emph{$\mucard$-multiuniversal} if for any $M \in \K$, any $A \subseteq M$, and any $b \in \cl^M (A)$, $\gtp (b / A; M)$ is $\eta$-algebraic. Except in some remarks, this more general notion will not be used in the rest of the paper, but note that:

  \begin{enumerate}
  \item Any universal class is a $2$-multiuniversal AEC. Conversely, a canonical expansion of the vocabulary of a $2$-multiuniversal AEC makes it into a universal class. See \cite[Theorem 2.8]{multipres-pams} for a precise statement.
  \item A multiuniversal class is exactly an $\aleph_0$-multiuniversal class.
  \item Any AEC $\K$ with intersections is $\LS (\K)^+$-multiuniversal. Indeed, if $b \in \cl^M (A)$ then by finite character (Fact \ref{inter-basics}) there exists a finite $A_0 \subseteq A$ with $b \in \cl^M (A_0)$, and the latter structure has cardinality at most $\LS (\K)$.
  \end{enumerate}
\end{remark}

\begin{example}\label{mu-example} \
  \begin{enumerate}
  \item Any universal class is a multiuniversal AEC.
  \item Let $T$ be a first-order theory with quantifier elimination. Let $\K$ be the class of all algebraically closed subsets of models of $T$, ordered by being a substructure. Then $\K$ is a multiuniversal AEC.
  \item Any AEC with intersections in which the closure operator is \emph{locally finite} (i.e.\ the closure of any finite set is finite) is a multiuniversal AEC, see \cite{locally-finite-aec-jsl} for a discussion of locally finite AECs.
  \item The AEC $\K$ of algebraically closed fields (ordered by subfield) is a multiuniversal AEC which is not a universal class. 

  \item\label{mu-example-6} Let $K$ be the class of pairs $(A, E)$, where $E$ is an equivalence relation on $A$, each of whose classes are countably infinite. Order it by the relation ``equivalence classes do not grow''. The resulting AEC $\K$ is sometimes called the \emph{toy quasiminimal class}. It is easy to check that $\K$ has intersections but is not $\aleph_0$-multiuniversal.
  \item Let $\K$ be the class of algebraically closed valued fields of rank one (that is, the valuation embeds into a subgroup of the reals). We code this by adding a constant symbol for every real number in the vocabulary. Then $\K$ is a multiuniversal AEC. Moreover, $\K$ is not axiomatizable by a first-order theory.

  \item The FCA classes introduced in \cite{aec-field-aut-v1} are AECs of relatively algebraically closed fields with commuting automorphisms. They are multiuniversal, but generally not first-order axiomatizable.
  \item Recall that a graph is \emph{locally finite} if all its vertices have finite degree. Let $K$ be the class of locally finite graphs, and make it into an AEC $\K$ by ordering it with $G \lea H$ if and only if $G$ is a subgraph of $H$ and if $v$ is in $H$ and there is an edge from $v$ to $G$, then $v \in G$ (that is, any connected component of $G$ in $H$ is $G$). Note that $\K$ admits intersections, is not first-order-axiomatizable, and is multiuniversal. Indeed, let $M \in \K$, $A \subseteq M$, and $b \in \cl^M (A)$. Let $A_0 \subseteq A$ be finite such that $b \in \cl^M (A_0)$ (Fact \ref{inter-basics}). It suffices to see that $p := \gtp (b / A_0; M)$ is $\aleph_0$-algebraic. Note that $b$ is connected to $A_0$ and the length $n$ of the smallest path from $A_0$ to $b$ is part of the information carried by the type of $b$ over $A$. Now since $M$ is locally finite, there are finitely-many vertices in $M$ at distance at most $n$ from $A_0$, and only those could realize $p$, so $p$ is indeed $\aleph_0$-algebraic.
  \item Let $\tau$ consist of unary predicates $P$ and $Q$, of a binary relation $E$, and of binary relations $\seq{R_m : m \in [2, \omega)}$. We let $K$ be the class of $\tau$-structures $M$ such that:

    \begin{enumerate}
    \item $P^M \cup Q^M = M$, $P^M \cap Q^M = \emptyset$, $E^M \subseteq P^M \times Q^M$. We think of $Q^M$ as a set of subsets of $P^M$ and of $E^M$ as being the membership relation.
    \item For each $s \in Q^M$, there exists at least one but only finitely-many $x \in P^M$ such that $x E^M s$. We let $n (s)$ be $|\{x \in P^M \mid x E^M s\}|$. Intuitively, $Q^M$ consists of finite non-empty subsets of $P^M$, and $n(s)$ denotes the cardinality of the set coded by $s$.
    \item If $x E^M s$ and $x E^M s'$, then $s = s'$. That is, there is at most one set containing each element.
    \item For each $x \in P^M$, there exists $s \in Q^M$ such that $x E^M s$. That is, each element is contained in at least one set. This and the previous axioms imply that $Q^M$ codes a partition of $P^M$ consisting of finite sets.
    \item For each $m \in [2, \omega)$, $R_{m}$ is the graph of a bijection from $\{s \in Q^M \mid n (s) = 1\}$ onto $\{s \in Q^M \mid n (s) = m\}$. Thus $R_{m}$ witnesses that in the partition there are as many sets with $m$ elements as with one element.
    \end{enumerate}

    We make $K$ into an AEC $\K$ by ordering it by $M \lea N$ if and only if $M \subseteq N$, and the sets do not grow: $x E^N s$ implies that $x \in N \backslash M$ if and only if $s \in N \backslash M$ (in particular the value of $n(s)$ is the same in $N$ and $M$). Then $\K$ has intersections and $\K$ is a multiuniversal AEC (because, in a nutshell, $\cl^M (A)$ adds only one set of each cardinality, and each such set has finitely-many elements). Moreover $\K$ is categorical in every infinite cardinal and is not first-order axiomatizable. Note also that $\K$ is not a universal class (because for a fixed $s$ the elements $x$ such that $x E s$ all have the same type).
    
  \item Let $\K$ be a universal class of abelian groups such that each group has finitely many $n$-torsion elements for every $n < \omega$ (and non-trivial $n$-torsion for at least one $n$). For example, the class of groups of the form $G \times \mathbb{Z}_n$, where $G$ is abelian torsion-free and $\mathbb{Z}_n$ is the cyclic group of order $n$ (coded by constant symbols in the language).

    Let $\K^{div}$ be the class consisting of the injective envelopes (generated divisible extensions) of groups in $\K$.  Then $\K^{div}$ is not a universal class because there is no function that takes an element $g$ to its divisors.  However, the number of values for $\frac{g}{n}$ is precisely equal to the amount of $n$-torsion in the group.  Thus, $\K^{div}$ is a multiuniversal class.
  
  If the original class $\K$ instead had $n$-torsion that was bounded above by $\mucard$, then the resulting class $\K$ would be $\mucard^+$-multiuniversal.
  \end{enumerate}
\end{example}

\section{Some model theory of multiuniversal AECs}

\subsection{Tameness and shortness in multiuniversal AECs}

It is natural to ask when a Galois type is determined by its restriction to small subsets. One can also ask when a type of a sequence is determined by its restriction to small subsequences. In the literature, the former property is called \emph{tameness} \cite{tamenessone} and the latter is called \emph{type-shortness} \cite[Definition 3.3]{tamelc-jsl}. We amalgamate these two properties into one definition here:

\begin{defin}
  For $\kappa$ an infinite cardinal, an AEC $\K$ is \emph{$(<\kappa)$-short} if for any $N_1, N_2 \in \K$ and any ordinal $\alpha$, if for $\ell = 1,2$, $\ba_\ell$ are sequences of length $\alpha$ from $N_\ell$ such that $\gtp (\ba_1 \rest I / \emptyset; N_1) = \gtp (\ba_2 \rest I / \emptyset; N_2)$ for all $I \subseteq \alpha$ of size strictly less than $\kappa$, then $\gtp (\ba_1 / \emptyset; N_1) = \gtp (\ba_2 / \emptyset; N_2)$.
\end{defin}
\begin{remark}\label{tameness-rmk}
  If $\K$ is $(<\kappa)$-short, then \cite[Theorem 3.5]{tamelc-jsl} $\K$ is fully $(<\kappa)$-tame and short over every set. That is, for any $N_1, N_2 \in \K$, any $A \subseteq N_1 \cap N_2$, and any ordinal $\alpha$, if for $\ell = 1,2$, $\ba_\ell$ are sequences of length $\alpha$ from $N_\ell$ such that $\gtp (\ba_1 \rest I / A_0; N_1) = \gtp (\ba_2 \rest I / A_0; N_2)$ for all $I \subseteq \alpha$ and all $A_0 \subseteq A$ both of size strictly less than $\kappa$, then $\gtp (\ba_1 / A; N_1) = \gtp (\ba_2 / A; N_2)$.
\end{remark}

The second author has shown that, for any AEC $\K$, if $\kappa > \LS (\K)$ is a strongly compact cardinal, then $\K$ is $(<\kappa)$-short \cite[Theorem 4.5]{tamelc-jsl}. Further, one can show in ZFC that universal classes are $(<\aleph_0)$-short (this appears as \cite[Theorem 3.7]{ap-universal-apal} but is due to the second author). We generalizes this result to multiuniversal classes here. The argument is more complex because an element $b \in \cl^M(A)$ cannot necessarily be written as a term with parameters from $A$. Instead, given $\ba_1$ and $\ba_2$ that locally have the same type, we may have several choices in how to put together the partial mappings witnessing this.  As we attempt to put these maps together, we might revise previous choices, but because of the algebraicity of types, the revision does not happen too many times. This rough intuition is formalized using Tychonoff's theorem.

\begin{thm}\label{tameness-thm}
  Any multiuniversal AEC is $(<\aleph_0)$-short.
\end{thm}
\begin{proof}
  Let $N_1, N_2 \in \K$, let $\ba_\ell$ be $\alpha$-sequences of elements from $N_\ell$, $\ell = 1,2$. Let $p_\ell := \gtp (\ba_\ell / \emptyset; N_\ell)$. Assume that $p_1^{I} = p_2^I$ for all finite $I \subseteq \alpha$ (here, $p_\ell^I$ denotes $\gtp (\ba_\ell \rest I / \emptyset; N_\ell$)). We have to show that $p_1 = p_2$. We will show that there is an isomorphism $f: \cl^{N_1} (\ba_1) \rightarrow \cl^{N_2} (\ba_2)$ so that $f (\ba_1) = \ba_2$. Let $M_\ell := \cl^{N_\ell} (\ba_\ell)$. For $\ell = 1,2$, let $A_\ell := \ran{\ba_\ell}$, and let $f_0 : A_1 \rightarrow A_2$ be the function sending $\ba_1$ to $\ba_2$.

  Call a partial function $f$ from $M_1$ to $M_2$ an \defn{$(M_1, M_2)$-mapping} if for some (any) enumeration $\bb$ of the domain of $f$, $\gtp (\bb / \emptyset; M_1) = \gtp (f (\bb) / \emptyset; M_2)$. Set
  \[
  P = \{B \subseteq M_1 \mid B \text{ is finite and } B \subseteq \cl^{M_1}(B \cap A_1)\}
  \]
  For each $B \in P$, let $F_B$ be the collection of $(M_1, M_2)$-mappings with domain $B$ that agree with $f_0$. By the assumption that $p_1^I = p_2^I$ for all finite $I$ and Fact \ref{inter-basics}, $F_B$ is not empty. By multiuniversality of $\K$, $F_B$ is finite, hence compact under the discrete topology. Thus, by Tychonoff's Theorem, $F := \prod_{B \in P} F_B$ is compact when endowed with the product topology. 
  
Now for each $B \in P$, let $N_B$ be the set of sequences $\seq{f_{B_0} : B_0 \in P}$ from $F$ that agree up to $B$: $f_{B_0} \subseteq f_B$ for all $B_0 \subseteq B$ in $P$. The set $N_B$ is not empty and is closed. Further, the collection $\{N_B : B \in P\}$ has the finite intersection property:  for $B_0, B_1, \ldots, B_{n - 1} \in P$, we have that $N_B \subseteq \bigcap_{i < n} N_{B_i}$ for any $B \in P$ extending $\bigcup_{i < n} B_i$. By the compactness of $F$, there is $\bar{f} = \seq{f_B : B \in P} \in \bigcap_{B \in P} N_B$. It is then straightforward to check that $\bigcup_{B\in P} f_B$ is an isomorphism from $M_1$ to $M_2$ that sends $\ba_1$ to $\ba_2$.
\end{proof}
\begin{remark}
  Theorem \ref{tameness-thm} does \emph{not} generalize to $\mucard$-multiuniversal classes: by Remark \ref{eta-multi-rmk}, any AEC $\K$ with intersection is $\LS (\K)^+$-universal, but there are numerous examples of non-tame AECs with intersections (see e.g.\ \cite{non-locality, lc-tame-pams}).
\end{remark}

\subsection{Abstract Morleyizations and compactness}

In \cite[Definition 3.3]{sv-infinitary-stability-afml}, the third author introduced the \emph{Galois Morleyization} of an AEC. It is an expansion of the vocabulary that adds predicates for each Galois types over the empty set (with length below a fixed bound).  Following this, we say that an AEC $\K$ is \defn{$(<\aleph_0)$-Morleyized} if for every $p \in \gS^{<\omega}_\K(\emptyset)$, there is a relation $R_p \in \tau(\K)$ (of the same arity as $p$) such that, for each $M \in \K$, $R_p(M) = p(M)$ (that is, $R_p$ is realized exactly by the elements realizing $p$ in $M$). By \cite[Proposition 3.5]{sv-infinitary-stability-afml}, each AEC has a functorial expansion (see \cite[Definition 3.1]{sv-infinitary-stability-afml}) to a $(<\aleph_0)$-Morleyized AEC. By \cite[Theorem 3.16]{sv-infinitary-stability-afml}, $\K$ is $(<\aleph_0)$-short if and only if Galois types are quantifier-free first-order types in the $(<\aleph_0)$-Galois Morleyization of $\K$. We give this conclusion a name:

\begin{defin}\label{syntactic-def}
  We say that an AEC $\K$ \emph{has quantifier-free types} if for any $N_1, N_2 \in \K$, $A \subseteq N_1 \cap N_2$ and (possibly infinite) sequences $\ba_\ell \in N_\ell$, $\ell = 1,2$, we have $\gtp (\ba_1 / A; N_1) = \gtp (\ba_2 / A; N_2)$ if and only if $\tpqf (\ba_1 / A; N_1) = \tpqf (\ba_2 / A; N_2)$. Here, $\tpqf (\ba / A; N)$ denotes the first-order quantifier-free type of $\ba$ over $A$ as computed in $N$.
\end{defin}

The discussion in the previous paragraph showed:

\begin{fact}\label{syntactic-fact}
  Any $(<\aleph_0)$-short $(<\aleph_0)$-Morleyized AEC has quantifier-free types.
\end{fact}

We now prove a compactness theorem for AECs with quantifier-free types. First, some terminology:

\begin{definition}
  Let $\K$ be an AEC and let $p (\bx)$ be a set of quantifier-free first-order formulas in $\tau (\K)$.

  \begin{enumerate}
  \item We say that $p$ is \defn{$\K$-satisfiable} if there is $M \in \K$ and $\ba \in M$ such that $M \models \phi[\ba]$ for all $\phi \in p$.
  \item We say that $p$ is \defn{strongly finitely $\K$-satisfiable} if for every finite subsequence $\bx_0$ of $\bx$, $p \rest \bx_0$ is $\K$-satisfiable, where $p \rest \bx_0$ denotes the set of formulas in $p$ whose free variables are all in $\bx_0$.
  \end{enumerate}

\end{definition}

\begin{thm}[Compactness for AECs with quantifier-free types]\label{compactness-thm}
  Let $\K$ be an AEC with quantifier-free types and intersections. Let $p (\bx)$ be a complete set of quantifier-free first-order formulas in $\tau (\K)$. If $p$ is strongly finitely $\K$-satisfiable, then $p$ is $\K$-satisfiable.
\end{thm}

Note that the essential flavor of this proof is the argument that local AECs are compact, see \cite[Lemma 11.5]{baldwinbook09}.

\begin{proof}
  Assume $\bx = (x_i)_{i < \theta}$ for $\theta$ a cardinal. We work by induction on $\theta$: assume without loss of generality that $\theta$ is infinite and we have the result for all $\theta_0 < \theta$. For $I \subseteq \theta$, write $p^I$ for the restriction of $p$ to formulas only using variables in $\{x_i \mid i \in I\}$, and $\gS^I (\emptyset)$ for the corresponding set of $I$-indexed Galois types. 

  \underline{Claim}: For each $I \subseteq \theta$ with $|I| < \theta$, there is a unique Galois type $q_I \in \gS^I (\emptyset)$ such that for any $M \in \K$ and $\bb \in M$, $\gtp (\bb / \emptyset; M) = q_I$ if and only if $M \models \phi[\bb]$ for all formulas $\phi \in p^I$.

  \underline{Proof of Claim}: When $I$ is finite, this holds by completeness and strong finite consistency, using that every Galois type is represented by a formula. When $I$ is infinite, use the induction hypothesis for existence and the equivalence between syntactic and Galois types for uniqueness. $\dagger_{\text{Claim}}$

  Now we build $\seq{M_i : i < \theta}$, $\seq{f_{i, j} : i \le j < \theta}$, and $\seq{\bb_i : i < \theta}$ such that for each $i \le j \le k < \theta$:

  \begin{enumerate}
  \item $M_i \in \K$, $f_{i, j} : M_i \rightarrow M_j$ is a $\K$-embedding, $\bb_i$ is an $i$-indexed sequence from $M_i$, $f_{i,j} (\bb_i) = \bb_j \rest i$.
  \item $f_{j,k} \circ f_{i, j} = f_{i, k}$, and $f_{i,i}$ is the identity.
  \item $\gtp (\bb_i / \emptyset; M_i) = q_i$.
  \end{enumerate}

  This is possible: we proceed inductively. When $i = 0$, this is easy. When $i = j + 1$ is a successor, we use the uniqueness of $q_j$ together with the characterization of equality of types in AECs with intersections (Fact \ref{inter-basics}). When $i$ is limit, we take colimits and use shortness. At the end, the colimit of the system gives the desired type.
\end{proof}
\begin{remark}
  The result also holds in case the AEC has amalgamation, or more generally if it has enough amalgamation to go through the successor steps at the end of the proof. Using the terminology of \cite[2.6]{existence-categ-v2}, $\K$ must have weak amalgamation (for long-enough types) and coherent sequences.
\end{remark}


\subsection{Model completeness}

Universal classes are examples of AECs whose $\K$-substructure relation is just substructure. Following the first-order definition, Baldwin and Kolesnikov \cite[Section 4]{bk-hs} called these AECs \emph{model complete}. Not all model complete AECs are universal classes (for example, algebraically closed fields are not universal). In fact, any time that the AEC is finitary (in the sense of Hyttinen and Kesälä \cite{finitary-aec}), we can expand the vocabulary to obtain a model complete AEC \cite[Theorem 3.15]{logic-intersection-bpas}. We will use the following special case in this paper:

\begin{lemma}\label{mc-fact}
  Any AEC with quantifier-free types is model complete.
\end{lemma}
\begin{proof}
  Let $\K$ be an AEC with quantifier-free types. Let $M, N \in \K$ be such that $M \subseteq N$. Let $\ba$ be an enumeration of $M$. It is enough to show that $\gtp (\ba / \emptyset; M) = \gtp (\ba / \emptyset; N)$ (then the definition of Galois types gives that $M \lea N$). But this holds because $M \subseteq N$ and quantifier-free types are the same as Galois types. 
\end{proof}

\subsection{Finitary isolation}

An interesting feature of multiuniversal classes is that types are isolated over finite subtypes:

\begin{defin}\label{isolates-def}
  Let $\K$ be an AEC and let $M \in \K$. Let $A, B \subseteq M$, let $\alpha$ be an ordinal, and let $p \in \gS^\alpha (A; M)$, $q \in \gS^\alpha (B; M)$. We say that \defn{$p$ isolates $q$ in $M$} if whenever $p = \gtp (\bb / A; M)$, we have that $q = \gtp (\bb / B; M)$.
\end{defin}

\begin{definition}\label{isolation-axiom}
  An AEC $\K$ with intersections has \defn{finitary isolation} if whenever $M \in \K$, $A \subseteq M$, and $\bb$ is a finite sequence of elements from $\cl^M (A)$, then there exists $A_0 \subseteq A$ finite such that $\gtp (\bb / A_0; M)$ isolates $\gtp (\bb / A; M)$ in $M$. 
\end{definition}
\begin{remark}
  In the terminology of \cite[5.1]{quasimin-five}, $\K$ has finitary isolation if $\gtp (\bb / A; M)$ is $s$-isolated whenever $\bb$ is a finite sequence from $\cl^M (A)$.
\end{remark}

\begin{thm}\label{isolation-thm}
  Any multiuniversal class has finitary isolation.
\end{thm}
\begin{proof}
  Let $\K$ be a multiuniversal class. Let $M \in \K$, $A \subseteq M$, and $\bb$ be a finite sequence of elements in $\cl^M (A)$. Let $p := \gtp (\bb / A; M)$. Fix $A_0 \subseteq A$ finite such that $\bb \in \cl^M (A_0)$. By definition, $p \rest A_0$ is $\aleph_0$-algebraic, hence by Lemma \ref{intersec-alg} is $(n + 1)$-algebraic for some $n < \omega$. Let $\bb_0, \ldots, \bb_{n - 1}$ be all the realizations of $p \rest A_0$ inside $M$. Now pick $A_1 \subseteq A$ finite such that $A_0 \subseteq A_1$ and for any $i < j < n$, $\gtp (\bb_i / A; M) \neq \gtp (\bb_j / A; M)$ implies that $\gtp (\bb_i / A_1; M) \neq \gtp (\bb_j / A_1; M)$. This is possible by shortness (Theorem \ref{tameness-thm} and Remark \ref{tameness-rmk}). Now it is easy to check that $p \rest A_1$ isolates $p$ in $M$.
\end{proof}

\subsection{Topological and category-theoretic characterizations of multiuniversal classes}\label{top-sec}

While this is not needed for the rest of the paper, it is interesting to note that multiuniversal classes can be characterized in terms of the compactness of certain automorphism groups (we say that a group $G$ of automorphism of a structure $M_0$ is \emph{compact} if it is compact in the product space $\fct{M_0}{M_0}$, where $M_0$ itself is given the discrete topology). The proof is standard, so we omit it. See for example section 4.1 of \cite{hom-struct} for some hints.

\begin{thm}\label{top-charact}
  Let $\K$ be an AEC with intersections. The following are equivalent:

  \begin{enumerate}
  \item\label{topo-1} $\K$ is multiuniversal.
  \item\label{topo-2} For any $M \in \K$ and any $A \subseteq M$, $\Aut_A (\cl^M (A))$ is compact.
  \item\label{topo-3} For any $M \in \K$ and any \emph{finite} $A \subseteq M$, $\Aut_A (\cl^M (A))$ is compact.
  \end{enumerate}
\end{thm}

We deduce a purely category-theoretic characterization of multiuniversal classes. Recall that a group $G$ is \emph{profinite} if it is an inverse limit of finite groups. Equivalently, \cite[Corollary 1.2.4]{wilson-profinite} $G$ can be topologized so that it is compact, Hausdorff, and totally disconnected. Note that the former definition of profinite is completely algebraic, so whether a group is profinite does not depend on how it is topologized. On general grounds, the automorphism groups of structures discussed in the previous theorem are always Hausdorff and totally disconnected (any two distinct points are separated by clopen sets). Thus they are compact if and only if they are profinite. We can use profinite groups to characterize what kind of accessible categories are multiuniversal classes (we thank Jiří Rosický for suggesting profinite groups as the solution to this problem). We briefly sketch how, but assume a good understanding of the related characterizations in \cite{multipres-pams}.

The AECs with intersections are exactly the locally $\aleph_0$-polypresentable categories with all morphisms monos \cite[5.7]{multipres-pams}. The locally polypresentable categories are defined using the concept of a polyinitial object: a set $\mathcal{I}$ of objects such that for any object $M$ of the category there is a unique object $i \in \mathcal{I}$ having a morphism $i \to M$, and for any $f, g: i \to M$, there is a unique automorphism $h: i \to i$ with $fh = g$. Call such a polyinitial object $\mathcal{I}$ \emph{profinite} if for each $i \in \mathcal{I}$, the group of automorphisms of $i$ is profinite. For example, the category of algebraically closed fields has a profinite polyinitial object (consisting of the algebraic closures of the prime fields for each characteristic). Using the concept of a profinite polyinitial object, we can then define the notion of a locally profinite-polypresentable category, and prove that multiuniversal AECs correspond to locally $\aleph_0$-profinite-polypresentable categories with all morphisms monos.

\section{Eventual categoricity}

The following notation for certain threshold cardinals that come up often in the theory of AECs will be used:

\begin{definition}\label{hanf-def}
 For a cardinal $\mu$, we set $\hanf{\mu} := \beth_{\left(2^{\mu}\right)^+}$. There will be a cardinal, $\hanf{\K}$ associated to an AEC which will be important. We refer the reader to \cite[Definition 2.16]{categ-universal-2-selecta} for a precise definition but note this cardinal satisfies $\beth_{\LS (\K)^+} \le \hanf{\K} \le \hanf{\LS (\K)}$.
\end{definition}

The goal of this section is to prove the main result of the paper:

\begin{thm}\label{main-thm}
  Let $\K$ be an AEC. Assume that:

  \begin{enumerate}
  \item $\K$ has intersections (Definition \ref{inter-def}).
  \item $\K$ has quantifier-free types (Definition \ref{syntactic-def}).
  \item $\K$ has finitary isolation (Definition \ref{isolation-axiom}).
  \end{enumerate}

  If $\K$ is categorical in \emph{some} $\lambda \ge \beth_{\hanf{\K}}$, then $\K$ is categorical in \emph{all} $\lambda' \ge \beth_{\hanf{\K}}$.
\end{thm}

Note that universal classes are AECs with quantifier-free types and intersections with finitary isolation (intersections hold by \cite[Example 2.6(1)]{ap-universal-apal}, having quantifier-free types holds by \cite[Remark 3.8]{ap-universal-apal}, and finitary isolation holds by Theorem \ref{isolation-thm}). Thus Theorem \ref{main-thm} generalizes the categoricity theorem for universal classes of the third author \cite[Theorem 7.3]{categ-universal-2-selecta}. In fact the proof is essentially the same, but one has to check that everything still goes through.

Theorem \ref{main-thm} also gives us our desired categoricity transfer for multiuniversal classes as a corollary of our previous analysis.

\begin{cor}\label{main-cor}
  Let $\K$ be a multiuniversal AEC. If $\K$ is categorical in \emph{some} $\lambda \ge \beth_{\hanf{2^{\LS (\K)}}}$, then $\K$ is categorical in \emph{all} $\lambda' \ge \beth_{\hanf{2^{\LS (\K)}}}$.
\end{cor}
\begin{proof}
  By definition, $\K$ has intersections. By Theorem \ref{tameness-thm}, $\K$ is $(<\aleph_0)$-short and by Theorem \ref{isolation-thm}, $\K$ has finitary isolation. Morleyize by adding at most $2^{\LS (\K)}$-relation symbols for Galois types of finite length (see \cite[\S3]{sv-infinitary-stability-afml}). We obtain an AEC with Löwenheim-Skolem-Tarski number at most $2^{\LS (\K)}$ which satisfies the hypotheses of Theorem \ref{main-thm} and the Moleyization preserves categoricity.
\end{proof}
\begin{remark}
  It is worth noting that there are AECs which satisfy the conditions of Theorem \ref{main-thm} but are not multiuniversal. In particular, Example \ref{mu-example}(\ref{mu-example-6}) is one such AEC.
\end{remark}

The proof of Theorem \ref{main-thm} goes along the lines of the corresponding result for universal classes proven in \cite[Theorem 7.3]{categ-universal-2-selecta}. We will quote and use terminology freely from \cite{categ-universal-2-selecta} (particularly Section 6). The reader is advised to have a copy open on their desk as they read this proof, although we will try to give a sense of the definitions used. The proof traces back to checking that some arguments of Shelah from \cite{sh300-orig} (or the revised version from \cite[Chapter V]{shelahaecbook2}) still go through in the setup of this paper. We may occasionally have to cite from \cite[Chapter V]{shelahaecbook2}, as \cite{categ-universal-2-selecta} relies on it.

A crucial tool in this work is the use of averages.  Recall that we have moved to a context where Galois types are quantifier-free.  This allows us to make the following definition: given a sequence $\seq{\bb_i : i \in I}$, a model $M$, a subset $A \subseteq M$ and a cardinal $\chi$, $\Av_\chi \left(\seq{\bb_i : i \in I} / A; M\right)$ is the set of quantifier-free first-order formulas $\phi (\bx, \ba)$ such that $M \models \phi[\bb_i, \ba]$ for all but fewer than $\chi$-many $i$'s. As opposed to the first-order case, $\Av_\chi \left(\seq{\bb_i : i \in I} / A; M\right)$ is not necessarily a complete type. We call the sequences $\seq{\bb_i : i \in I}$ whose average \emph{does} yield a complete type \emph{convergent} (after suppressing some parameters). The main technique for finding convergent sequences is \cite[Theorem V.A.2.8]{shelahaecbook2}, which says that, if the class fails to have an appropriate order property, then any sufficiently large set of parameters can be `pruned' to a convergent subset of the same size.  The assumption of no order property follows from categoricity by \cite[Lemma 7.1]{categ-universal-2-selecta}.

It suffices to prove \cite[Fact 6.10]{categ-universal-2-selecta} for classes satisfying the assumptions of Theorem \ref{main-thm}. Then the rest of the proof of \cite[Theorem 7.3]{categ-universal-2-selecta} is exactly the same. More precisely, we will show:

\begin{thm}\label{selecta-variation}
  Let $\K$ be an AEC. Assume that:
  
  \begin{enumerate}
  \item $\K$ has intersections.
  \item $\K$ has quantifier-free types.
  \item $\K$ has finitary isolation.
  \item $\chi \ge \LS (\K)$ is such that $\K$ does not have the order property of length $\chi^+$.

  \end{enumerate}

  Set $\mu := 2^{2^{\chi}}$ and let $\K^0 := (K, \le^{\chi^+, \mu^+})$ ($\le^{\chi^+, \mu^+}$ is the ordering defined in \cite[Definition 6.7]{categ-universal-2-selecta}: roughly it requires that $M \subseteq N$, and the type of any $\bb \in {}^{<\omega}N$ over $M$ is the average of some sequence from $M$). Then:

  \begin{enumerate}
  \item $\K^0$ is a weak AEC with $\LS (\K^0) \le \mu^+$ (this means it satisfies the axioms of an AEC except perhaps smoothness of unions).
  \item Let $\cl^M$ be the closure operator on $\K$ and let $\nf$ be the $4$-ary relation defined in the statement of \cite[Fact 6.10]{categ-universal-2-selecta}; the key condition is that $M_1\nf_{M_0}^{M_3} M_2$ implies that the Galois type of any $\ba \in M_1$ over $M_2$ is the average of a sequence from $M_0$.\\
  We then have that $(\K^0, \nf, \cl)$ satisfies $\Axfr$ (see \cite[Definition 4.1]{categ-universal-2-selecta}). Moreover, 
  \begin{enumerate}
      \item $\cl$ is algebraic: $\cl^M (A) = \cl^N (A)$ whenever $M \subseteq N$ (see \cite[Definition 5.22]{categ-universal-2-selecta}); and
      \item $\nf$ is $\mu^+$-based (see \cite[Definition 4.12]{categ-universal-2-selecta}).
    \end{enumerate}
  \end{enumerate}
\end{thm}
\begin{proof}[Proof of Theorem \ref{main-thm}]
    By the proof of \cite[Theorem 7.3]{categ-universal-2-selecta}, substituting Theorem \ref{selecta-variation} to \cite[Fact 6.10]{categ-universal-2-selecta}.
\end{proof}
\begin{proof}[Proof of Theorem \ref{selecta-variation}]
  That $\K^0$ is a weak AEC with $\LS (\K^0) \le \mu^+$ is as in the proof of \cite[6.10]{categ-universal-2-selecta}. Now $\cl$ is algebraic because $\K$ is model complete (Lemma \ref{mc-fact}). That $\nf$ is $\mu^+$-based is proven exactly as in the proof of \cite[Fact 6.10]{categ-universal-2-selecta}. It remains to see that $(\K^0, \nf, \cl)$ satisfies $\Axfr$. From now on, we write ``convergent'' instead of ``$(\chi^+, \mu^+)$-convergent'', ``averageable'' instead of ``$(\chi^+, \mu^+)$-averageable'', and ``$\Av (\BI / A; M)$'' instead of ``$\Av_{\chi^+} (\BI /A ; M)$''. Two key claims will be:

  \underline{Claim 1}: Let $M_0, M \in K$ with $M_0 \subseteq M$. Let $\ba, \bb \in M$ be finite sequences such that $\bb \in \cl^M (\ba)$ and $\gtp (\bb / \ba; M)$ isolates $\gtp (\bb / M_0; M)$ in $M$. If $\tpqf (\ba / M_0; M)$ is averageable over $M_0$ in $M$, then so is $\tpqf (\bb / M_0; M)$.

  \underline{Proof of Claim 1}: Note that Galois and quantifier-free types are interchangeable by hypothesis. Let $\BI := \seq{\ba_i : i \in I}$ be a sequence of elements in $M_0$ of the same arity as $\ba$ which is convergent and such that $\Av (\BI / M_0; M) = \tpqf (\ba / M_0; M)$. Since $\BI$ has at least $\mu^+$-many elements, we can prune it further to assume without loss of generality that $\tpqf (\ba_i / \emptyset; M) = \tpqf (\ba / \emptyset; M)$ for all $i \in I$. Thus for each $i \in I$ there exists an isomorphism $f_i : \cl^{M_0} (\ba_i) \rightarrow \cl^{M} (\ba)$ sending $\ba_i$ to $\ba$. For $i \in I$, let $\bb_i := f_i^{-1} (\bb)$. Let $\BJ := \seq{\bb_i : i \in I}$. By pruning, we can assume without loss of generality that $\BJ$ is convergent. We want to see that $\Av (\BJ / M_0; M) = \tpqf (\bb / M_0; M)$.

  Let $\phi (\bx; \bc) \in \Av (\BJ / M_0; M)$. This means that for all but at most $\chi$-many $i \in I$, $M \models \phi[\bb_i; \bc]$. Let $p := \tpqf (\ba \bb / \emptyset; M)$. By construction of $\bb_i$, $p = \tpqf (\ba_i \bb_i / \emptyset; M)$ for all $i \in I$. Now by averageability, for most $i \in I$, $\tpqf (\ba_i / \bc; M) = \tpqf (\ba / \bc; M)$. Let $g_i : \cl^M (\ba_i) \rightarrow \cl^M (\ba)$ be an isomorphism sending $\ba_i\bc$ to $\ba \bc$, and let $\bb'_i := g_i (\bb_i)$. Then for most $i$, $M \models \phi[\bb'_i; \bc] \land p[\bb'_i; \ba]$. Thus $\gtp (\bb / \ba; M) = \gtp (\bb'_i / \ba; M)$, but by isolation this means that $\gtp (\bb'_i / M_0; M) = \gtp (\bb / M_0; M)$, so in particular $M \models \phi[\bb; \bc]$, as desired. $\dagger_{\text{Claim 1}}$

  \underline{Claim 2}: If $\nfs{M_0}{M_1}{M_2}{M_3}$, then $M_2 \leap{\K^0} M_3$

  \underline{Proof of Claim 2}: By definition of $\nf$ and transitivity of $\leap{\K^0}$ we can assume without loss of generality that $M_3 = \cl^{M_3} (M_1 \cup M_2)$. Let $\bb \in M_3$ be a finite sequence. We have to see that $\tpqf (\bb / M_2; M_3)$ is averageable over $M_2$. By finitary isolation, there is a finite $A \subseteq M_1 \cup M_2$ such that $\tpqf (\bb / M_1 \cup M_2; M_3)$ is isolated by $\tpqf (\bb / A; M_3)$ in $M_3$. In particular, $\tpqf (\bb / M_2)$ is isolated by $\tpqf (\bb / A; M_3)$ in $M_3$. Let $\ba$ be an enumeration of $A$. By Claim 1, it suffices to see that $\tpqf (\ba / M_2; M_3)$ is averageable over $M_2$ in $M_3$. Now by definition of $\nf$, $\tpqf (\ba / M_2; M_3)$ is averageable over $M_0$, hence over $M_2$, as desired. $\dagger_{\text{Claim 2}}$

  We now prove all the axioms from the definition of $\Axfr$ in \cite[Definition 4.1]{categ-universal-2-selecta}. Most of them have either been observed before or (like finite character) follow from Lemma \ref{inter-basics}. It suffices to show existence, uniqueness, symmetry, and base enlargement. 

  \begin{enumerate}
  \item \underline{Existence}: Let $M_0 \leap{\K^0} M_\ell$, $\ell = 1,2$. First, we build $M_3 \in \K^0$ and $\K$-embeddings $f_\ell : M_\ell \rightarrow M_3$ for $\ell = 1,2$ such that 
  \begin{enumerate}
    \item $f_1 \rest M_0 = f_2 \rest M_0$;
    \item $\nfs{f_1[M_0]}{f_1[M_1]}{f_2[M_2]}{M_3}$; and 
    \item $M_3 = \cl^M_3 (f_1[M_1] \cup f_2[M_2])$.
    \end{enumerate}
    Note that we do not claim that the $f_\ell$'s are $\K^0$-embeddings, although this will follow from Claim 2 and Symmetry below.
    
    Let $\lambda := |M_1| + |M_2| + \aleph_0$. For $\ell = 1,2$, let $\bc^\ell := \seq{c_i^\ell : i < \lambda}$ be an enumeration (possibly with repetitions) of $M_\ell$. For $u \subseteq \lambda$, write $\bc_u^\ell$ for $\bc^\ell \rest u$. For each finite $u \subseteq \lambda$, by definition of $\leap{\K^0}$, there is a convergent sequence $\BI_u$ inside $M_0$ such that $\tpqf (\bc_u^1 / M_0; M_1) = \Av (\BI_u / M_0; M_1)$. Let $q_u := \Av (\BI_u / M_2; M_2)$, seen as a type in the variables $\bx_u^1 := \seq{x_i^1 : i \in u}$. Now let $p_u$ be the set of quantifier-free formulas $\phi (\bx_u^1; \bx_u^2)$ such that $\phi (\bx_u^1; \bc_u^2) \in q_u$. Let $p := \bigcup_{u \in [\lambda]^{<\aleph_0}} p_u$.  Note that $p$ is complete as a quantifier-free type.

    Now, for each finite $u \subseteq \lambda$, $p_u$ contains at most $(|\tau (\K)| + \aleph_0)$-many formulas and $\BI_u$ has the much bigger size $\mu^+$. Moreover for each $\phi \in p_u$, all but fewer than $\mu^+$-many elements of $\BI_u$ satisfy $\phi (\bx_u^1; \bc_u^2)$. It follows that $p_u$ is realized in $M_2$. Thus $p$ is strongly finitely $\K$-satisfiable. By Theorem \ref{compactness-thm}, $p$ is $\K$-satisfiable. Let $\bd^1 \bd^2$ realize $p$ (where $\bd^\ell$ realize $p \rest \bx^\ell$) inside some $M \in \K$. Let $M_3 := \cl^M (\bd^1 \bd^2)$. Now the formula ``$\bx_u^1 = \bx_v^2$'' is in $p$ whenever $\bc_u^1 = \bc_v^2$ are in $M_0$. Moreover, $\tpqf (\bd^\ell / \emptyset; M_3) = \tpqf (\bc^\ell / \emptyset; M_1)$. Since Galois and quantifier-free types are the same (in $\K$), sending $\bc^\ell$ to $\bd^\ell$ is a $\K$-embedding $f_\ell : M_\ell \rightarrow M_3$ and $f_1 \rest M_0 = f_2 \rest M_0$. By construction, $\nfs{f_1[M_0]}{f_1[M_1]}{f_2[M_2]}{M_3}$.
  \item \underline{Uniqueness}: We first prove a uniqueness statement for types, as stated in \cite[Claim V.A.4.6(2)]{shelahaecbook2}, then use the fact that quantifier-free types are the same as Galois types as well as the standard argument in \cite[Lemma 12.6]{indep-aec-apal} to get uniqueness in the sense of amalgams.
  \item \underline{Symmetry}: No change from the proof of \cite[Fact 6.10]{categ-universal-2-selecta}, see \cite[Sublemma V.B.2.11]{shelahaecbook2}.

  \item \underline{Base enlargement}: Assume $\nfs{M_0}{M_1}{M_2}{M_3}$ and $M_0 \leap{\K^0} M_2' \leap{\K^0} M_2$. We want to see that $\nfs{M_2'}{\cl^{M_3} (M_2' \cup M_1)}{M_2}{M_3}$. Now by monotonicity we know that $\nfs{M_0}{M_1}{M_2'}{M_3}$, hence by definition of $\nf$ and the previous part $M_2' \leap{\K^0} \cl^{M_3} (M_2' \cup M_1)$. Also, $\cl^{M_3} (\cl^{M_3} (M_2' \cup M_1) \cup M_2) = \cl^{M_3} (M_1 \cup M_2) \leap{\K^0} M_3$ by definition of $\nf$ and the assumption that $\nfs{M_0}{M_1}{M_2}{M_3}$. It remains to see that for any finite sequence $\bc \in \cl^{M_3} (M_2' \cup M_1)$, $\tpqf (\bc / M_2; M_3)$ is averageable over $M_2'$.

    First we show the transitivity property of $\nf$: if $\nfs{N_0}{N_1}{N_2}{N_3}$ and $\nfs{N_2}{N_3}{N_4}{N_5}$, then $\nfs{N_0}{N_1}{N_4}{N_5}$. To see this, let $\bb \in N_1$ be a finite sequence. We want to see that $\tp (\bb / N_4; N_5)$ is averageable over $N_0$, and we know that $\tp (\bb / N_4; N_5)$ is averageable over $N_2$ and $\tp (\bb / N_2; N_5)$ is averageable over $N_0$. To conclude what we want, imitate the standard ``forking calculus'' proof of transitivity from extension and uniqueness (see \cite[Claim II.2.18]{shelahaecbook}), noting that the usual base monotonicity property is trivial for the notion of being averageable over. 

    Now that we have transitivity, we can conclude base enlargement on general grounds using forking calculus: as in the proof of \cite[Claim III.9.6(E)(b)]{shelahaecbook}, there is $M_1'$ and $M_3'$ such that $M_3 \leap{\K^0} M_3'$, $M_1 \lea M_1'$, and $\nfs{M_2'}{M_1'}{M_2}{M_3'}$. Now observe that $\cl^{M_3'} (M_1 \cup M_2') = \cl^{M_3} (M_1 \cup M_2')$ to conclude.
  \end{enumerate}
\end{proof}

\bibliographystyle{amsalpha}
\bibliography{categ-multi}

\end{document}